\documentclass[12pt, reqno]{amsart}
\usepackage{amsmath, amsthm, amscd, amsfonts, amssymb, graphicx, color}
\usepackage[bookmarksnumbered, colorlinks, plainpages]{hyperref}

\newtheorem{theorem}{Theorem}[section]
\newtheorem{lemma}[theorem]{Lemma}
\newtheorem{proposition}[theorem]{Proposition}
\newtheorem{corollary}[theorem]{Corollary}
\theoremstyle{definition}

\newtheorem{example}[theorem]{Example}

\theoremstyle{remark}
\newtheorem{remark}[theorem]{Remark}
\numberwithin{equation}{section}

\begin{document}

\title[Classes of operators related to  $m$-isometric operators ]
{ Classes of operators related to   $m$-isometric operators}
\author[S. Mecheri and O. A. M. Sid Ahmed]{Salah Mecheri and  Sid Ahmed Ould Ahmed Mahmoud  }
\address{S. Mecheri\endgraf
  Department of Mathematics\endgraf
Tebessa University  \endgraf
 12002-Tebessa Algeria }
 \email{mecherisalah@hotmail.com}
\address{ Sid Ahmed Ould Ahmed Mahmoud  \endgraf
  Mathematics Department, College of Science, Jouf
University,\endgraf Sakaka P.O.Box 2014. Saudi Arabia}
\email{sidahmed@ju.edu.sa}
\begin{abstract}
 Isometries played a pivotal role in the development of operator theory, in particular with the theory
of contractions and polar decompositions and has been widely studied due to its fundamental importance in the theory of stochastic processes,
the intrinsic problem of modeling the general contractive operator via its isometric dilation and many other areas in applied mathematics.
In this paper we present some properties of $n$-quasi-$(m,C)$-isometric operators. We show that a power of a $n$-quasi-$(m, C)$-isometric operator is again
a $n$-quasi-$(m,C)$-isometric operator and some products and tensor products of
$n$-quasi-$(m,C)$-isometries are again n-quasi-$(m,C)$-isometries.
\end{abstract}

\maketitle

\section{Introduction }

\vspace*{-1.5pt}
\quad  Let $\mathcal{H}$ be a separable infinite dimensional complex  Hilbert space with inner product $\langle .\;| \;.\rangle$,  $\mathcal{B}(\mathcal{H})$  be the set of all bounded linear operators on $\mathcal{H}$, and $I=I_\mathcal{H}$ be the identity operator. For every $T\in \mathcal{B}(\mathcal{H})$  its range is denoted by  $\mathcal{R}(T)$, its null space by $\mathcal{N}(T)$. The adjoint of $T$ is denoted  by $T^*$. A subspace $\mathcal{M}\subset \mathcal{H}$   is invariant for $T$ (or $T$-invariant) if $T\mathcal{M}\subset \mathcal{M}$. As usual, the orthogonal complement and  the closure of $\mathcal{M}$ are denoted $\mathcal{M}^\perp$ and $\overline{\mathcal{M}}$ respectively. We denote by $P_\mathcal{M}$ the orthogonal projection on $\mathcal{M}$.
\par \vskip 0.2 cm
\noindent A conjugation is a conjugate-linear operator $C : \mathcal{H} \longrightarrow \mathcal{ H}$, which is both involutive
(i.e., $C^2 = I )$ and isometric (i.e., $\langle Cx\;|\;Cy\rangle=\langle y\;|\;x\rangle  \quad (\forall\; x, y \in \mathcal{H})$).

 \par \vskip 0.2 cm
\noindent
 Recall that if $C$ is a conjugation on $\mathcal{H}$, then $\|C\|=1$, $\big(CTC\big)^k=CT^kC$ and $\big( CTC\big)^*=CT^*C$ for every positive integer $k$ (see \cite{GP} and \cite{GPP}
for more details).\par \vskip 0.2 cm
\noindent
Throughout this paper, let $m$ and $n$ be natural numbers. An operator
$T \in \mathcal{B}(\mathcal{H})$ is said to be :\par \vskip 0.2 cm
\noindent
 \;$m$-isometry
 if \begin{equation}\label{(1.1)}
\displaystyle\sum_{0\leq
k \leq m}(-1)^{m-k}\binom{m}{k}T^{*m-k}T^{m-k}=0,
\end{equation}
or equivalently if
\begin{equation}
\sum_{0\leq k\leq m}(-1)^{m-k}\binom{m}{k}\|T^kx\|^2=0\;\;\quad\forall\;\;x\in{\mathcal H}.
\end{equation}
\noindent where $\displaystyle \binom{m}{k}$ is the binomial coefficient. These class of operators have been introduced  and studied by J. Agler and M. stankus in \cite{AS1}, \cite{AS2} and \cite{AS3}.
In recent years, the $m$-isometric operators have received substantial
attention.  It has been proved in \cite{BMM} and  \cite{BMN2} that
the powers of an $m$-isometry are $m$-isometries and some products of $m$-isometries
are again $m$-isometries. On the other hand, the perturbation of $m$-isometries by
nilpotent operators has been considered in \cite{BMN1}, \cite{TMVN}, \cite{FYH} and the dynamics of $m$-isometries
has been explored in \cite{FB} and other papers. Furthermore, Duggal studied
the tensor product of $m$-isometries in \cite{BD}.  In addition, $m$-isometry weighted shift operators have been discussed
in \cite{ABL} and the reference therein. S. Mecheri and T.Parasad  in \cite{MP} extended
the notion of $m$-isometric operator to the case of $n$-quasi-$m$-isometric operators  of
bounded linear operators on a Hilbert space. An operator $T\in {\mathcal B}({\mathcal H})$ is said to be
 $n$-quasi-$m$-isometric operator  if
\begin{eqnarray}\label{(1.3)}
T^{*n}\bigg(\sum_{0\leq k\leq m}(-1)^k\binom{m}{k}T^{*m-k}T^{m-k}\bigg)T^n=0,
\end{eqnarray}
\noindent The $1$-quasi-isometries are shortly called quasi-isometries, such
operators being firstly studied  in \cite{SP1} and \cite{SP2}.\par\vskip 0.2 cm \noindent
In \cite{CKL}, M. Ch\={o}, E. Ko and J. Lee introduced
$(m,C)$-isometric operators with conjugation $C$   and studied properties of such operators. For an operator $T \in {\mathcal B}({\mathcal H}$) and an integer $ m \geq 1$ ,
$T$ is said to be an $(m,C)$-isometric operator if there exists some conjugation $C$ such that
\begin{equation}
\sum_{0\leq k\leq m}(-1)^{k}\binom{m}{k}T^{*m-k}CT^{m-k}C=0.
\end{equation}
According to definitions of $m$-isometry, $n$-quasi-$m$-isometry and $(m,C)$-isometry, The authors in \cite{OCL} define an $n$-quasi-$(m,C)$-isometry
$ T$ as follows. An operator $T$ is said to be an $n$-quasi-$(m,C)$-isometric operator if there exists some conjugation
$C$ such that
\begin{equation}
T^{*n}\bigg(\sum_{0\leq k\leq m}(-1)^{k}\binom{m}{k}T^{*m-k}CT^{m-k}C\bigg)T^n=0.
\end{equation}
It is easy to see that the class of $n$-quasi-$(m,C)$-isometry contains every  $(m,C)$-isometric operators  with conjugation $C$.In general, this inclusion relation is  proper (see \cite{OCL}). Many
results  about the class of $n$-quasi-$(m,C)$-isometric operators have been found in \cite{OCL}.\par \vskip 0.2 cm \noindent
In this paper it is
shown that the operators in this class have many interesting properties in common with  $m$-isometries, $n$-quasi-$m$-isometries  and $(m,C)$-isometric operators.
In particular, we show that the powers of an $n$-quasi-$(m,C)$-isometry are  $n$-quasi-$(m,C)$-isometries and some products  and tensor products of $n$-quasi-$(m, C)$-isometries
are again $n$-quasi-$(m,C)$-isometries. It has also been proved that
the sum of an $n$-quasi-$(m,C)$-isometry and a commuting nilpotent operator of degree $p$ is a
 $(n+p)$-quasi-$(m + 2p- 2)$-isometry.

\section{Main Results}

We begin by the following theorem, which is a
structure theorem for $n$-quasi-$(m, C)$-isometric operators.\par \vskip 0.2 cm \noindent
In \cite{OCL}, the authors studied the matrix representation of $n$-quasi-$(m, C)$-isometric  operator
with respect to the direct sum of $\overline{{\mathcal R}(T^n)}$ and its orthogonal complement. In the
following we give an equivalent condition for $T$ to be $n$-quasi-$(m,C)$-isometric operator.
\begin{theorem} \label{th21} Let $C=C_1\oplus C_2$ be a conjugation on  ${\mathcal H}$ where $C_1$ and $C_2$ are conjugation on $\overline{{\mathcal R}(T^n)}$ and ${\mathcal N}({T^{\ast}}^n)$, respectively. Assume that  ${\mathcal R}(T^n)$ is not dense, then the following statements are equivalent:\par \vskip 0.2 cm \noindent

$(1)$ $T$ is $n$-quasi-$(m,C)$-isometric operator, \par \vskip 0.2 cm \noindent

$(2)$  $T=\left(
         \begin{array}{ccc}
           T_1 & T_2 \\
          0 & T_3
         \end{array}
       \right)$ on ${\mathcal H}=\overline{{\mathcal R}(T^n)}\oplus {\mathcal N}({T^{\ast}}^n)$, where $T_1$ is an $(m,C_1)$-isometric operator on $\overline{{\mathcal R}(T^n)}$, $T_3^n=0$, and $\sigma(T)=\sigma(T_1)\cup\{0\}$ where $\sigma(T)$ is the spectrum of $T$.
\end{theorem}

\begin{proof} $(1)\Rightarrow (2)$. Consider the matrix representation of $T$ with respect to the decomposition ${\mathcal H}=\overline{{\mathcal R}(T^n)}\oplus {\mathcal N}({T^{\ast}}^n)$:
$$T=\left(
         \begin{array}{ccc}
           T_1 & T_2 \\
          0 & T_3
         \end{array}
       \right)\ \mbox{on}\ {\mathcal H}=\overline{{\mathcal R}(T^n)}\oplus {\mathcal N}({T^{\ast}}^n).$$
Let $P$ be the projection of ${\mathcal H}$ onto $\overline{{\mathcal R}(T^n)}.$
Since  $T$ is an $n$-quasi-$(m,C)$-isometric operator, it follows that
$$P\bigg(\sum_{0\leq k\leq m}(-1)^k\binom{m}{k}T^{*m-k}CT^{m-k}C\bigg)P=0.$$
This means that $$\sum_{0\leq k\leq m}(-1)^k\binom{m}{k}T_1^{*m-k}C_1T_1^{m-k}C_1=0.$$
Hence $T_1$ is  an $(m,C_1)$-isometric operator on $\overline{{\mathcal R}(T^n)}$.
Let $x=x_1\oplus x_2\in\overline{{\mathcal R}(T^n)}\oplus {\mathcal N}({T^{\ast}}^n)={\mathcal{H}}$. If $x\in {\mathcal N}({T^{\ast}}^n)$, then
\begin{eqnarray*}
\langle {T_{3}}^{n}x_2,x_2\rangle&=& \langle T^{n}(I-P)x,(I-P)x \rangle\cr
&=&\langle (I-P)x, {T^{\ast}}^{n}(I-P)x \rangle=0.
\end{eqnarray*}
Hence ${T_{3}}^{n}=0$.
So, we get that $\sigma(T)=\sigma(T_1)\cup\{0\}$.\par \vskip 0.2 cm \noindent

$(2)\Rightarrow(1)$ Suppose that $$ T = \begin{pmatrix} T_{1} & T_{2} \\
                            0 & T_{3} \end{pmatrix} \quad \text{ on } \quad
                              \mathcal{H} = \overline{\mathcal{R}(T^{n})}  \oplus
                              \mathcal{N}(T^{*n}) $$
                                where $ \overline{{\mathcal R}(T^{n})} $ is the closure
                                of $\mathcal{R}(T^{n})$, where
                                  $T_{1}$ is an $(m, C_1)$-isometry and
                                 $T_{3}^{k} = 0$: Since
 $$ T^{ k}= \begin{pmatrix} T^{n}_{1} & \displaystyle\sum_{0\leq j\leq k-1}T^{j}_{1}T_{2}T^{k-1-j}_{3} \\
                            0 & 0 \end{pmatrix}$$
we have
$$T^{*n}(\sum_{0\leq l\leq m} (-1)^{k}\begin{pmatrix} m \\
  k \end{pmatrix}T^{* m-k}CT^{m-k}C)T^{n}=$$
  $$  \begin{pmatrix} T_{1} & T_{2} \\
                            0 & T_{3} \end{pmatrix}^{*n}
                            \bigg(\sum_{0\leq k\leq m}(-1)^{k}\begin{pmatrix} m \\
  k \end{pmatrix}\begin{pmatrix} T_{1} & T_{2} \\
                            0 & T_{3} \end{pmatrix}^{* m-k} \begin{pmatrix}
                                                                C_1 & 0 \\
                                                                0 & C_2 \\
                                                              \end{pmatrix}
                                                             \begin{pmatrix} T_{1} & T_{2} \\
                            0 & T_{3} \end{pmatrix}^{m-k}\begin{pmatrix}
                                                                C_1 & 0 \\
                                                                0 & C_2 \\
                                                              \end{pmatrix}
                                                            \bigg)\begin{pmatrix} T_{1} & T_{2} \\
                            0 & T_{3} \end{pmatrix}^{*n}$$
                            $$=\begin{pmatrix} T^{*n}_{1}DT^{k}_{1} & T^{*n}_{1}D\displaystyle\sum_{0\leq j \leq k-1}T^{j}_{1}T_{2}T^{k-1-j}_{3} \\
  \bigg(\displaystyle\sum_{0\leq j\leq k-1}T^{j}_{1}T_{2}T^{k-1-j}_{3}\bigg)^{*}DT^{k}_{1} & \bigg(\displaystyle\sum_{0\leq j\leq k-1}T^{j}_{1}T_{2}T^{k-1-j}_{3}\bigg)^{*}D\sum_{0\leq j\leq k-1}T^{j}_{1}T_{2}T^{k-1-j}_{3}
  \end{pmatrix},$$ where
  $$ D= \sum_{0\leq k \leq m} (-1)^{k}\begin{pmatrix} m \\
  k \end{pmatrix}T_1^{* m-k}C_1T^{m-k}C_1.$$ This implies that $T^{*n}\bigg(\displaystyle\sum_{0\leq k \leq m}(-1)^{k}\begin{pmatrix} m \\
  k \end{pmatrix}T^{* m-k}CT^{m-k}C\bigg)T^{n}=0$ on $\mathcal{H} = \overline{\mathcal{R}(T^{n})}  \oplus
                              \mathcal{N}(T^{*n}) $. Thus $T$ is $n$-quasi-$(m,C)$-isometric operator.

\end{proof}

\begin{corollary} \label{cor21} If $T$ is a $n$-quasi-$(m, C)$-isometric operator and ${\mathcal R}(T^{n})$ is dense, then $T$
is a $(m, C)$-isometric operator.
\end{corollary}
In \cite{CKL}, the authors showed that a power of an $(m, C)$-isometric operator is again a $(m, C)$-isometric operator.
In the following theorem we show that this remains true for $n$-quasi-$(m, C)$-isometric operators.
\begin{theorem} \label{th22}
Let $C=C_1\oplus C_2$ be a conjugation on  ${\mathcal H}$ where $C_1$ and $C_2$ are conjugation on $\overline{{\mathcal R}(T^n)}$ and ${\mathcal N}({T^{\ast}}^n)$, respectively.
If $T$ is a $n$-quasi-$(m, C)$-isometric operator, then so is $T^{k}$ for every
natural number $k$.
\end{theorem}
\begin{proof} If ${\mathcal R}(T^n)$ is dense, then $T$ is an $(m,C)$-isometric operator and so is $T^k$ for every positive integer $k$.\par \vskip 0.2 cm \noindent If ${\mathcal R}(T^n)$ is not  dense. By Theorem \ref{th21} we write the matrix representation of $T$ on ${\mathcal H} =\overline{{\mathcal R}(T^n)}\oplus {\mathcal N} (T^{*n})$ as follows
 $$ T = \begin{pmatrix} T_{1} & T_{2} \\
                            0 & T_{3} \end{pmatrix} \quad \text{ on } \quad
                              \mathcal{H} = \overline{\mathcal{R}(T^{n})}  \oplus
                              \mathcal{N}(T^{*n}) $$ where $T_1$ is an $(m,C_1)$-isometric operator. By
 \cite[ Theorem 2.1]{CKL}, $T^{k}_{1}$ is an $(m, C_1)$-isometric operator. Since
$$T^{k}=\begin{pmatrix} T^{k}_{1} & \displaystyle\sum_{0\leq j\leq k-1}T^{j}_{1}T_{2}T^{k-1-j}_{3} \\
                            0 & T^{k}_{3} \end{pmatrix} \,\;\;\mbox{on}\; \mathcal{H} = \overline{\mathcal{R}(T^{n})}  \oplus
                              \mathcal{N}(T^{*n}). $$ Thus $T^{k}$ for every natural number $k$ is a $n$-quasi-$(m, C)$-isometric operator by Theorem 2.1.
\end{proof}
\begin{remark}
       The converse of Theorem \ref{th22} in not true in general as shown in the following example.
       \end{remark}

\begin{example}
Let $C$ be a conjugation on $\mathbb{C}^2 $ defined by $ C(x_1, x_2) = ( \overline{x_2},-\overline{x_1})$ and consider the
operator matrix $T=\left(
                     \begin{array}{cc}
                       -1 & - 1\\
                       3 & 2 \\
                     \end{array}
                   \right)
$ on $\mathbb{C}^2$. A simple calculation shows that $T^{\ast3}\bigg(T^{*3}CT^3C-I\bigg)T^3 = 0$ and $T^*\bigg(T^*CTC-I\bigg)T\not=0$. So, we obtain
that $T^3$ is a quasi-$(1,C)$-isometric operator, but $T$ it is not a quasi- $(1,C)$-isometric operator.
\end{example}

\par \vskip 0.2 cm \noindent It was observed that every $(m,C)$-isometric operator is an $(k,C))$-isometric operator for every integer $k\geq m.$ In the following proposition we show that this remains true for  $n$-quasi-$(m,C)$-isometric operator.

\begin{proposition} \label{pro21}

Let $T\in {\mathcal B}({\mathcal H})$ and
 let $C=C_1\oplus C_2$ be a conjugation on  ${\mathcal H}$ where $C_1$ and $C_2$ are conjugation on $\overline{{\mathcal R}(T^n)}$ and ${\mathcal N}({T^{\ast}}^n)$, respectively.
 If $T$ is an $n$-quasi-$(m, C)$-isometric operator, then $T$ is an $l$-quasi-$(k, C)$-isometric operator for every positive integers $k\geq m$ and $l\geq n$.\end{proposition}
\begin{proof}
If ${\mathcal R}(T^n)$ is dense, then $T$ is an $(m,C)$-isometric operator and hence $T$ is an $(k,C)$-isometric operator for every positive integer $k\geq m$.\par \vskip 0.2 cm \noindent If ${\mathcal R}(T^n)$ is not dense, by Theorem \ref{th21} we write the matrix representation of T on \\${\mathcal H} =\overline{{\mathcal R}(T^n)}\oplus {\mathcal N} (T^{*n})$ as follows $T=\left(
         \begin{array}{ccc}
           T_1 & T_2 \\
          0 & T_3
         \end{array}
       \right)$ where $T_1=T_{/\overline{{\mathcal R}(T^n)}}$ is an $(m,C_1)$-isometric operator and $T_3^n=0$. Obviously that $T_1$ is an $(k,C_1)$-isometric operator for every integer $k\geq m$. The conclusion follows from the statement $(2)$ of Theorem \ref{th21}.

\end{proof}
\par \vskip 0.1 cm \noindent
For an operator $ T \in {\mathcal B}({\mathcal H})$ and a conjugation $C$, the operator $\Lambda_m(T)$ is define  by
$$\Lambda_m(T):=\sum_{0\leq k \leq m}(-1)^k\binom{m}{k}T^{*m-k}CT^{m-k}C.$$ Then $T$ is an $(m,C)$-isometric operator if and only if $\Lambda_m(T)=0.$ \par \vskip 0.2 cm \noindent
The following lemma gives another condition for which an $n$-quasi-$(m,C)$-isometric operator became an $n$-quasi-$(k,C)$-isometric operator for $k\geq m$.
\begin{lemma}\label{lem21}
Let $T \in {\mathcal B}({\mathcal H})$ be an $n$-quasi-$(m,C)$-isometric operator where $C$ is a conjugation on ${\mathcal H}$. It $T(CTC)=(CTC)T$, then $T$ is an $n$-quasi-$(k,C)$-isometric operator for every positive integer $k\geq m$.
\end{lemma}
\begin{proof}
It is well known that $\Lambda_{m+1}(T)=T^*\Lambda_m(T)(CTC)-\Lambda_m(T)$ (\cite{CKL}). Under the assumptions that  $T$ is an $n$-quasi-$(m,C)$-isometric operator and satisfies $T(CTC)=(CTC)T$, it follows
$$T^{*n}\Lambda_{m+1}(T)T^n=T^{*n+1}\Lambda_m(T)T^n(CTC)-T^{*n}\Lambda_m(T)T^n=0.$$
Therefore $T$ is an $n$-quasi-$(m+1,C)$-isometric operator.
\end{proof}
\par \vskip 0.02 cm \noindent
Let $T \in {\mathcal B}({\mathcal H} )$. Denote by $r(T)$ the spectral radius of $T$, that
is, \\$r(T) = \max\{\;|\lambda| : \lambda \in  \sigma(T) \;\}.$  We say that T is normaloid if $r(T) = \|T\|$.\par \vskip 0.2 cm \noindent
\begin{theorem}\label{th23}
Let $C=C_1\oplus C_2$ be a conjugation on ${\mathcal H}$ where $C_1$ and $C_2$ are  conjugation on $\overline{{\mathcal R}(T^n)}$ and $ {\mathcal N}(T^{*n})$ respectively. Let $T\in {\mathcal B}({\mathcal H})$ be an $n$-quasi-$(m, C)$- isometric operator. Assume that  $T$ is power bounded and $T_1=T_{/{\overline{{\mathcal R}(T^n)}}}$ satisfies $T_1C_1T_1C_1-I$ is normaloid, then $T$ is an $n$-quasi-$(1,C)$-isometric operator.

\end{theorem}
\begin{proof} We know that  $T$ admits the following matrix representation
 $T=\left(
                                               \begin{array}{cc}
                                                 T_1 & T_2 \\
                                                0 & T_3 \\
                                               \end{array}
                                             \right)$ on ${\mathcal H}= \overline{{\mathcal R}(T^n)} \oplus {\mathcal N}(T^{*n}).$
Since $T$ is an $n$-quasi-$(m,C)$-isometric operator, it follows in view of Theorem \ref{th21} that $T_1$ is an $(m,C_1)$-isometric operator and $T_3^n=0$. Furthermore $T$ is power bounded then it is easy that $T_1$ is power bounded and satisfies $T_1C_1T_1C_1-I$ is normaloid. By applying \cite[Theorem 3.1]{CKL} we obtain that $T_1$ is an $(1,C_1)$-isometric operator.
According to Theorem \ref{th21} we can deduce that $T$ is an $n$-quasi-$(1,C)$-isometric operator. Thus we complete the proof.
\end{proof}
\begin{lemma}(\cite[Lemma 3.15]{HYM})\label{lem22}
 If $(a_j)_j$ is a sequence of complex numbers and $r,s, m,l$ are positive
integers satisfying
\begin{equation}
\sum_{0\leq k\leq m}(-1)^{k}\binom{m}{k}a_{rk+j}=0
\end{equation}
and
\begin{equation}
\sum_{0\leq k\leq l}(-1)^{k}\binom{l}{k}a_{sk+j}=0
\end{equation}
for all $j\geq 0$,then
\begin{equation}
\sum_{0\leq k\leq q}(-1)^{k}\binom{q}{k}a_{pk}=0,
\end{equation}
where $q$ is the greatest common divisor of $r$ and $s$, and $p$ is the minimum of $m$ and $l$.
\end{lemma}
In \cite{BMM} it was proved that if $T^r$ is an $m$-isometry and $T^s$ is an $l$-isometry, then $T^q$ is a $p$-isometry,  where $q$ is the greatest common divisor of
$r$ and $s$, and $p$ is the minimum of $m$ and $l$. In the following theorem we extend this result as follows
\begin{theorem}\label{th24}
 Let $T \in {\mathcal B}({\mathcal H} )$ such that $T^r$ is an $(m,C)$-isometry and $T^s$ is an $(l,C)$-isometry, then $T^q$ is a $(p,C)$-isometry, where $q$ is the greatest common divisor of
$r$ and $s$, and $p$ is the minimum of $m$ and $l$.
\end{theorem}
\begin{proof}
\begin{eqnarray*}
T \;\text{is an}\; (m,C)-\text{isometry} &\Leftrightarrow& \sum_{0\leq k \leq m}(-1)^{m-k}\binom{m}{k}T^{*k}CT^kC=0\\&\Leftrightarrow& \sum_{0\leq k \leq m}(-1)^{m-k}\binom{m}{k}T^{*k}CT^k=0\\&\Leftrightarrow&
\sum_{0\leq k \leq m}(-1)^{m-k}\binom{m}{k}\langle CT^k\;|\;T^kx\rangle=0\;\;\; \forall\;x\in {\mathcal H}
\end{eqnarray*}
\par \vskip 0.2 cm \noindent
Fix $x \in {\mathcal H}$ and denote $a_j=\langle CT^jx\;|\;T^jx\rangle$ for $j=1,2,\cdots.$
 As $T^r$
is an $(m,C)$-isometric operator
the sequence $(a_j)_{j\geq 0}$ verifies the recursive equation $$\displaystyle\sum_{0\leq k\leq m}(-1)^{m-k}\binom{m}{k}a_{kr+j}=0, \;\text{for all}\; j\geq 0.$$ Analogously, as $T^s$
is an $(l,C)$-isometric operator
the sequence $(a_j)_{j\geq 0}$ verifies the recursive equation $$\displaystyle\sum_{0\leq k\leq l}(-1)^{l-k}\binom{l}{k}a_{kr+j}=0,\;\;\text{ for all}\;\; j\geq 0.$$ Applying  \label{lem22}similar  we obtain that
$$\displaystyle\sum_{0\leq k\leq q}(-1)^{q-k}\binom{q}{k}a_{kp}=0,$$  where $q$ is the greatest common
divisor of $r$ and $s$, and $p$ is the minimum of $m$ and $l$.
Finally $T^q$ is an $(p,C)$-isometric operator.
\end{proof}
\par\vskip 0.1 cm \noindent
The following corollary is
direct consequence of preceding theorem.

\begin{corollary}\label{cor22}
Let $T \in {\mathcal B}({\mathcal H})$ and let $r, s,m,l$ be
positive integers. The following properties hold. \par \vskip 0.2 cm \noindent $(1)$
If $T$ is an $(m,C)$-isometric operator  such that $T^s$ is an $(1,C)$-isometric

operator, then $T$ is an $(1,C)$-isometric
operator.\par \vskip 0.2 cm \noindent $(2)$
 If $T^r$ and $T^{r+1}$ are $(m,C)$-isometries, then so is $T$.\par \vskip 0.2 cm \noindent $(3)$
 If $T^r$ is an $(m,C)$-isometric operator and $T^{r+1}$ is an $(l,C)$-isometric operator

  with $m < l$, then $T$ is an $(m,C)$-isometric operator.
\end{corollary}
\par\vskip 0.2 cm \noindent
\begin{theorem} \label{th25}Let $S$ and $T$ be in ${\mathcal B}({\mathcal H})$ and let $C=C_1\oplus C_2$ be a conjugation on ${\mathcal H}$where $C_1$ and $C_2$ are conjugation on $\overline{{\mathcal R}(S^n)}$ and ${\mathcal N}({S^{\ast}}^n)$, respectively. Assume that $T$ and $S$ are doubly commuting and  $T(CTC)=(CTC)T$, $T(CSC)=S(CTC)$ and $S^{\ast}CTC=CTCS^{\ast}$.
If $T$ is an $n_1$-quasi-$(k,C)$-isometric operator and $S$ is an $n_2$-quasi-$(m,C)$-isometric operator, then   $TS$ is a $n^\prime=\max\{\;n_1,n_2\;\}$-quasi-$(k+m-1,C)$-isometric operator.-
\end{theorem}

\begin{proof}
Since $TS=ST$, $T(CSC)=S(CTC)$ and $S^{\ast}CTC=CTCS^{\ast}$, it follows that
$$[T^*, S^*]=[CTC,CSC]=[CTC,S^*]=0.$$  By taking into account  \cite[Lemma 12]{CG3} we obtain that
\begin{eqnarray*}
 \Lambda_{k+m-1}(TS)
=  \sum_{0\leq j \leq k+m-1} \binom{k+m-1}{j}  \, {T^{\ast}}^j T^{*{n^\prime}}\Lambda_{k+m-1-j}(T) T^{n^\prime}\,CT^jC  S^{*{n^\prime}}\Lambda_j(S) S^{n^\prime} .
\end{eqnarray*}
\noindent
Furthermore as $[T, S^*]=[T,CSC]=[CTC,S^*]=0$ we get
\begin{eqnarray*}
&&(TS)^{*n^{\prime}} \, \Lambda_{k+m-1}(TS) \, (TS)^{n^\prime} \cr
&=& \sum_{0\leq j \leq k+m-1} \binom{k+m-1}{j}  \, {T^{\ast}}^j T^{*{n^\prime}}\Lambda_{k+m-1-j}(T) T^{n^\prime}\,CT^jC  S^{*{n^\prime}}\Lambda_j(S) S^{n^\prime} .
\end{eqnarray*}
  Under the assumption that $S$ is an $n_2$-quasi-$(k,C)$-isometric operator, we get in view of Proposition \ref{pro21}    $S^{*{n^\prime}}\Lambda_j(S)S^{n^\prime}=0$  for $j\geq m$ and $n^\prime \geq n_2$. On the other hand, if $j\leq m-1$, then $k+m-1-j\geq k+m-1-m+1=k$ and so $  T^{*^{n^\prime}}\Lambda_{k+m-1-j}(T)T^{n^\prime}=0$ by Lemma \ref{lem21}.
Hence, $TS$ is a $n^\prime$-quasi-$(k+m-1,C)$-isometric operator.
\end{proof}
\par \vskip 0.2 cm \noindent
\par\vskip 0.2 cm \noindent
\begin{corollary} \label{cor23}Let $S$ and $T$ be in ${\mathcal B}({\mathcal H})$  are doubly commuting.  Let $C=C_1\oplus C_2$ be a conjugation on ${\mathcal H}$ where $C_1$ and $C_2$ are conjugation on $\overline{{\mathcal R}(S^n)}$ and ${\mathcal N}({S^{\ast}}^n)$, respectively. Assume that $T(CSC)=(CSC)T$, $T(CTC)=(CTC)T$ and $S^{\ast}CTC=CTCS^{\ast}$.
If $T$ is an $n_1$-quasi-$(k,C)$-isometric operator and $S$ is an $n_2$-quasi-$(m,C)$-isometric operator, then   $TS^q$ is a $n^\prime=\max\{n_1,n_2\}$-quasi-$(k+m-1,C)$-isometric operator for some positive integer $q$.
\end{corollary}
\begin{proof}
In view of Theorem \ref{th22} we have that $S^q$ is an $n_2$-quasi-$(m,C)$-isometric operator. Moreover $T$ and $S^q$  satisfy the conditions of Theorem \ref{th24}. Hence $TS^q$ is a $n^\prime$-quasi-$(k+m-1,C)$-isometric operator.
\end{proof}
\par \vskip 0.2 cm \noindent
\begin{proposition}\label{pro25}
Let $S$ and $T$ be in ${\mathcal B}({\mathcal H})$  are doubly commuting. Assume that  $T(CSC)=(CSC)T$, $T(CTC)=(CTC)T$, $S^{\ast}CTC=CTCS^{\ast}$ and
$S(CSC)=(CSC)S.$
If $T$ is an $n_1$-quasi-$(k,C)$-isometric operator and $S$ is an $n_2$-quasi-$(m,C)$-isometric operator, then   $TS$ is a $n^\prime=\max\{n_1,n_2\}$-quasi-$(k+m-1,C)$-isometric operator.
\end{proposition}
\begin{proof}
Under the assumptions that $T(CTC)=(CTC)T$ and $S(CSC)=(CSC)S$, it follows form Lemma \ref{lem21} that $T$ is an $n_1$-quasi-$(k+1,C)$-isometric operator and $S$ is an $n_2$-quasi-$(m+1,C)$-isometric operator. By  repeating the reasoning given  in the proof of Theorem \ref{th25} we check that
$$\big(TS\big)^{*n^{\prime}}\Lambda_{m+k-1}(TS)\big(TS\big)^{n^\prime}=0.$$ Therefore $TS$ is a  $n^\prime$-quasi-$(m+k-1)$-isometric operator.
\end{proof}
\vspace{4mm}
Let ${\mathcal H} \overline{\otimes} {\mathcal H}$ denote the completion,
endowed with a reasonable uniform cross-norm, of the algebraic tensor product ${\mathcal H} \otimes {\mathcal H}$
of ${\mathcal H}$ and ${\mathcal H}$. It is well known that if $x\in {\mathcal H} \overline{\otimes} {\mathcal H}$, there exists linearly independent sets $(u_1)_{i\in I}$ and
$(v_i)_{i\in I}$  such that $x=\displaystyle\sum_{i\in I}u_i\otimes v_i.$
An inner product on ${\mathcal H} \overline{\otimes} {\mathcal H}$ is defines as
$$\langle x\otimes y\;|\;u\otimes v\rangle:=\langle x\;|\;u\rangle\langle y\;|\;v\rangle\;\;\text{where}\;\;x,y, u, v \in {\mathcal H}.$$
We construct an operator $\widetilde{T}$ on the tensor product of Hilbert spaces. Let $T$ be an operator on ${\mathcal H}$ and $S$ be an operator
on ${\mathcal H}$. We define $$\widetilde{T}:=T\otimes S:{\mathcal H} \overline{\otimes} {\mathcal H}\longrightarrow {\mathcal H} \overline{\otimes} {\mathcal H}\;\;\text{by}$$
$$\widetilde{T}(x)=\big(T\otimes S\big)\bigg( \sum_{i\in I}u_i\otimes v_i \bigg)=\sum_{i\in I}T(u_i)\otimes S(v_i).$$
\par \vskip 0.2 cm \noindent
In \cite[Lemma 4.5]{CLH}, it was proved that if $C$ and $D$ be conjugations on ${\mathcal H}$. Then $C \otimes D$ is a conjugation on  ${\mathcal H}\overline{\otimes}  {\mathcal H}$.\par \vskip 0.1 cm \noindent

\begin{lemma}\label{lem23}
If $T\in {\mathcal B}({\mathcal H})$ and let $C$ and $ D$
are conjugations on $\mathcal{H}$ respectively. Then $T$ is an $n$-quasi-$(m,C)$-isometric operator  if and only if  then the tensor product $T\otimes I$
$(\text{resp.} I\otimes T )$ is  an $n$-quasi-$(m,C\otimes D)$-isometric operator.
\end{lemma}
\begin{proof}
A straightforward computation gives
\begin{eqnarray*}&&\big(T\otimes I\big)^{*n}\bigg(\sum_{0\leq k\leq m}(-1)^{m-k}\binom{m}{k}\big(T\otimes I\big)^{*k}\big(C\otimes D\big)\big(T\otimes I\big)^k\big(C\otimes D\big)\bigg)\big(T\otimes I\big)^n\\&=&T^{*n}\bigg(\sum_{0\leq k\leq m}(-1)^{m-k}\binom{m}{k}T^{*k}CT^kC\bigg)\otimes I.\end{eqnarray*}
From this we can  get that $T$ is an $n$-quasi-$(m,C)$-isometric operator if and only if  $T\otimes I$ is an $n$-quasi-$(m,C\otimes D)$-isometric operator..
\end{proof}

\par\vskip 0.1 cm \noindent

As  application of Lemma \ref{lem23} and Proposition \ref{pro25}, we get the following theorem.
\begin{theorem}\label{th26}
Let $T$ and $ S\in \mathcal{B}(\mathcal{H})$ such that  $T$ is an $n_1$-quasi-$(m,C)$-isometric operator  and $S$ is  an $n_2$-$(k,D)$-isometric operator where $C$ and $ D$
are conjugations on $\mathcal{H}$, respectively. If $T(CTC)=(CTC)T$ and  and $S(DSD)=(DSD)$, then $T\otimes S$ is an $n^\prime=\max\{n_1,n_2\}$-quasi-$(m+k-1,C\otimes D)$-isometric operator.
\end{theorem}
\begin{proof}
It is well known that  $T\otimes S=\big(T\otimes I\big)\big(I\otimes S\big)=\big(I\otimes S\big)\big(T\otimes I\big)$. In view of Lemma \ref{lem23} we have that
$T\otimes I$ is an $n_1$-quasi-$(m,C\otimes D)$-isometric operator and  $I\otimes S$ is an $n_2$-quasi-$(k,C\otimes D)$-isometric operator.
On the other hand, note  that $T\otimes I$ and $I\otimes S$ satisfy all conditions in
Proposition \ref{pro25}.  We conclude that $(T\otimes I)(I\otimes S)$ is a $n^\prime$-quasi-$(m+k-1,C\otimes D)$-isometric operator.
Thus we arrive at the desired
conclusion.
It needs to m
\end{proof}

\par\vskip 0.2 cm \noindent
\begin{lemma}\label{lem24}
Let $T,Q\in {\mathcal B}({\mathcal H})$ such that $TQ=QT$, then for $m\geq 2$
$$\Lambda_{m}(T+Q)=\sum_{i+j+k=m}\binom{m}{i,j,k}\big(T+Q\big)^{*i}Q^{*j}\Lambda_k(T)CT^jC CQ^iC$$ where
$\displaystyle\binom{m}{i,j,k}=\frac{m!}{i!\;j!\;k!}.$
\end{lemma}
\begin{proof} The proof follows by similar arguments as in the proof of
 \cite[Lemma 2]{OCL1}.
\end{proof}
\par \vskip 0.3 cm \noindent
It was proved in \cite[Thoerem 3.1]{TMVN} that if $T \in{\mathcal B}({\mathcal H})$ is an $m$-isometry and $Q \in{\mathcal B}({\mathcal H})$ is an
nilpotent operator of order $p$ such that $TQ = QT$, then $T +Q$-is an $(m+2p-2)$-isometry. In the following theorem we show that this
remains true for $(m,C)$-isometric operators.
\par \vskip 0.3 cm \noindent

\begin{theorem}\label{th27}
Let $ T, Q \in {\mathcal B}({\mathcal H})$. Assume $T$ commutes with $Q$. If $T$  is an $(m, C)$-isometric operator and $Q$
is a nilpotent operator of order $p$. Then $T+Q$ is an $(m +2p-2,C)$-isometric operator where $C$ is a conjugation on ${\mathcal H}$.
\end{theorem}
\begin{proof}
We need to show $$\Lambda_{m+2p-2}\big(T+Q\big)=0.$$
In view of  Lemma \ref{lem24}  we have
$$\Lambda_{m+2p-2}\big(T+Q\big)=\sum_{ i+j+k=m+2p-2}\binom{m+2p-2}{i, j,k}\big(T^*+Q^*\big)^iQ^{*j}\Lambda_{k}(T)CT^jCCQ^iC.$$
\noindent \rm(i) If $\max\{i, j\}\geq p$, then $CQ^iC=0$ or $Q^{*j}=0.$ \par \vskip 0.2 cm \noindent \rm(ii)  If $\max\{i, j\}\leq p-1,$ then $k\geq m$ and hence $\Lambda_k(T)=0.$ \par \vskip 0.2 cm \noindent From \rm(i) and \rm(ii) we get $\Lambda_{m+2p-2}\big(T+Q\big)=0.$
\end{proof}
\par \vskip 0.3 cm \noindent
In the following theorem we investigate the nilpotent perturbations of an $n$-quasi-$(m,C)$-isometric operator.
\par\vskip 0.3 cm \noindent
\begin{theorem}\label{th28}
Let $ T$ and $ Q \in {\mathcal B}({\mathcal H})$. Assume that $TQ=QT$ commutes, $TCQC=CQCT$ and $ TCTC=CTCT$ where $C$ is a conjugation on ${\mathcal H}$. If $T$  is an $n$-quasi-$(m, C)$-isometric operator and $Q$
is a nilpotent operator of order $p$. Then $T+Q$ is a $(n+p)$-quasi-$(m +2p-2,C)$-isometric operator.
\end{theorem}

\begin{proof} We need to show $$\big(T+Q\big)^{\gamma}\Lambda_{m+2p-2}\big(T+Q\big)\big(T+Q\big)^{\gamma}=0\;\text{where}\;  \gamma=n+p.$$
In view of  Lemma \ref{lem24}  we have
$$\Lambda_{m+2p-2}\big(T+Q\big)=\sum_{ i+j+k=m+2p-2}\binom{m+2p-2}{i, j,k}\big(T^*+Q^*\big)^iQ^{*j}\Lambda_{k}(T)CT^jCCQ^iC$$
and
\begin{eqnarray*}
&&\big(T+Q\big)^{*\gamma}\Lambda_{m+2p-2}\big(T+Q\big)\big(T+Q\big)^{\gamma}\\&=&
\bigg(\sum_{0\leq r\leq 2\gamma}\binom{\gamma}{r}T^{*(\gamma-r)}Q^{*r}\bigg)\bigg(\sum_{i+j+k=m+2p-2}\binom{m+2p-2}{i,j,k}\big(T^*+Q^*\big)^iQ^{*j}\Lambda_{k}(T)CT^jCCQ^iC \bigg)\\&&\times \bigg(\sum_{0\leq s\leq 2\gamma}\binom{\gamma}{s}T^{\gamma-s}Q^{s}\bigg).
\end{eqnarray*}
Now observe that if  $\max\{i, j\}\geq p$, then $CQ^jC=0$ or $Q^i=0.$
 and hence $$\big(T^*+Q^*\big)^iQ^{*j}\Lambda_{k}(T)CT^jC CQ^iC=0.$$
\noindent However , if  $\max\{ i, j\}\leq p-1$, then $k\geq m$.
Using the fact that $T$ is an $n$-quasi-$m$-isometry and $TCTC=CTCT$, we get
$$T^{*(n+p-r)}\Lambda_{k}(T)T^{n+p-s}=0\;\;\;\text{for}\;\;r\in \{\;0,\cdots,p\;\}\;\;\text{and}\;s \in \{\;0,\cdots,p \}$$  and
$$T^{*(n+p-r)}Q^{*r}\Lambda_k(T)T^{n+p-s}Q^s=0\;\;\;\text{for}\;\;r\in \{\;p+1,\cdots,n+p \} \;\;\text{and}\;s\in \{\;p+1,\cdots,n+p \}.$$

Combining the above arguments  we deduce that
$$\big(T+Q\big)^{n+p}\Lambda_{m+2p-2}\big(T+Q\big)\big(T+Q\big)^{n+p}=0.$$
Thus $T+Q$ is a $(n+p)$-quasi-$(m+2p-2)$-isometric operator. Therefore the theorem is proved.
\end{proof}

\begin{example}
Let $C$ be a conjugation on $\mathbb{C}^3 $ defined by $ C(x_1, x_2 ,x_3)= ( \overline{x_3},\overline{x_2},\overline{x_1})$ and consider the
operator matrix $T=\left(
                     \begin{array}{ccc}
                       1 & 0 & \alpha \\
                       0 & 1& 0 \\
                        0& 0 & 1 \\
                     \end{array}
                   \right)
$ on $\mathbb{C}^3$. Then $T=I+Q$. Since $Q^2=0$, it follows from Theorem \ref{th28}  that $T$ is a $(n+2)$-quasi-$(m+2,C)$-isometric operator.
\end{example}

\end{document}